\documentclass[12pt,reqno]{amsart}

\usepackage{amsmath, amsfonts, amssymb, amsthm, amscd, amsbsy}
\usepackage{fancyhdr}
\usepackage[usenames,dvipsnames,svgnames,x11names,hyperref]{xcolor}
\usepackage{geometry}
\usepackage{graphicx}
\usepackage[pagebackref]{hyperref}%
\usepackage{amsmath}%
\usepackage{amsfonts}%
\usepackage{amssymb}
\usepackage{enumerate}

\providecommand{\U}[1]{\protect\rule{.1in}{.1in}}
\hypersetup{backref=true,
pagebackref=true,
hyperindex=true,
colorlinks=true,
breaklinks=true,
urlcolor=NavyBlue,
linkcolor=Fuchsia,
bookmarks=true,
bookmarksopen=false,
filecolor=black,
citecolor=ForestGreen,
linkbordercolor=red
}

\allowdisplaybreaks

\newtheorem{theorem}{Theorem}[section]
\theoremstyle{plain}

\newtheorem{lemma}{Lemma}[section]

\numberwithin{equation}{section}

\def\vint{\mathop{\mathchoice%
          {\setbox0\hbox{$\displaystyle\intop$}\kern 0.22\wd0%
           \vcenter{\hrule width 0.6\wd0}\kern -0.82\wd0}%
          {\setbox0\hbox{$\textstyle\intop$}\kern 0.2\wd0%
           \vcenter{\hrule width 0.6\wd0}\kern -0.8\wd0}%
          {\setbox0\hbox{$\scriptstyle\intop$}\kern 0.2\wd0%
           \vcenter{\hrule width 0.6\wd0}\kern -0.8\wd0}%
          {\setbox0\hbox{$\scriptscriptstyle\intop$}\kern 0.2\wd0%
           \vcenter{\hrule width 0.6\wd0}\kern -0.8\wd0}}%
          \mathopen{}\int}

\begin{document}
\title[Sharp Stable inequalities on the Sphere $\mathbb{S}^n$]{Sharp Stability of Log-Sobolev  and Moser-Onofri inequalities on the Sphere}
\author{Lu Chen}
\address[Lu Chen]{ School of Mathematics and Statistics, Beijing Institute of Technology, Beijing
100081, PR China}
\email{chenlu5818804@163.com}

\author{Guozhen Lu}
\address[Guozhen Lu]{Department of Mathematics, University of Connecticut, Storrs, CT 06269, USA}
\email{guozhen.lu@uconn.edu}

\author{Hanli Tang}
\address[Hanli Tang]{Laboratory of Mathematics and Complex Systems (Ministry of Education), School of Mathematical Sciences, Beijing Normal University, Beijing, 100875, China}
\email{hltang@bnu.edu.cn}
\thanks{}


\begin{abstract}
In this paper, we are concerned with the stability problem for endpoint conformally invariant cases of the Sobolev inequality on the sphere $\mathbb{S}^n$. Namely, we will establish the stability for Beckner's log-Sobolev inequality and  Beckner's Moser-Onofri inequality on the sphere.
We also prove that the sharp constant of global stability for the log-Sobolev inequality on the sphere $\mathbb{S}^n$ must be strictly smaller than the sharp constant of local stability for the same inequality. Furthermore, we also derive the non-existence of the global stability for Moser-Onofri inequality on the sphere $\mathbb{S}^n$.
\end{abstract}

\maketitle
\section{Introduction}
In the past decades, there have been extensive works done in studying the stability of geometric and functional inequalities.
The stability problem essentially investigates the deficit  of a geometric and functional inequality in terms of a certain appropriate distance function from the class of the maximizers for which the equality holds. Simply put, we like to know how close the function is to the manifold consisting of all the maximizers if a function makes the geometric or functional inequality almost an equality. Sharp stability further studies the best constant
in front of the distance function.

\medskip

The main purpose of this paper is concerned with the problems of establishing the stability for the log-Sobolev and Moser-Onofri inequalities  on the sphere $\mathbb{S}^n$. Since the celebrated work of L. Gross on the log-Sobolev inequality in $\mathbb{R}^n$ with respect to the Gaussian measure \cite{Gross}, the log-Sobolev inequality on the sphere was established on the sphere $\mathbb{S}^n$ by Beckner \cite{Be1993} by using Fourier analysis and the spherical harmonics techniques. Furthermore, Beckner \cite{Be1993} also established the sharp Moser-Onofri inequality on the sphere $\mathbb{S}^n$ in higher dimensions using Fourier analysis techniques, among other things. (See also Carlen and Loss for another proof of Beckner's Moser-Onfri inequality on the sphere $\mathbb{S}^n$ using the competing symmetry and logarithmic Hardy-Littlewood-Sobolev inequality \cite{CarlenLoss}.)

\medskip

In this paper,
we will establish the sharp local stability for the log-Sobolev  inequality on the sphere $\mathbb{S}^n$ by choosing suitable $L^2$ distance function and applying the spectrum estimate of spherical harmonics.  (see Theorem \ref{Main Theorem 1}.) We further prove that the best constant for the global stability of the log-Sobolev inequality on the sphere $\mathbb{S}^n$ must be strictly smaller than the best constant for the local stability of the same inequality. (see Theorem \ref{thm2}.) The non-existence of the local stability of the log-Sobolev inequality on the sphere $\mathbb{S}^n$ with respect to another natural distance function is also established. (see Theorem \ref{Main Theorem 2}.) Furthermore, we will also prove the local stability of the Moser-Onofri inequality on the sphere $\mathbb{S}^n$ and the non-existence of the global stability of the Moser-Onofri inequality on the sphere $\mathbb{S}^n$. (see Theorem \ref{Main Theorem 3}.) Another interesting result in this paper is a
  relationship between the sharp local stability and global stability for some general geometric inequality satisfying certain homogeneity conditions. (See Theorem \ref{l-sta imply g-sta}.)  We will show that  when the remainder term of the little $o(d^p(u, M))$ in the local stability inequality holds for all the functions $u$ and the deficit of the geometric or functional inequality has the same homogeneity as the power of the distance function, then we can conclude the global stability inequality with the same best constant.  On the other hand, when the deficit of the geometric or functional inequality has a different power of the  homogeneity from  that of the distance function on the right hand side, then there does not exist global stability of the geometric or functional inequality. Indeed,    as an application of Theorem \ref{l-sta imply g-sta}, we show the nonexistence of the global stability of the Moser-Onofri inequality on the sphere $\mathbb{S}^n$.
  
  \medskip

  Stability of geometric and functional inequalities is a classical question in the study of geometric and functional inequalities and its history can date back to Brezis and Lieb's work in \cite{BrLi}.  It measures the distance to the set of maximizers in terms of the deficit functional. Sharp geometrical inequalities and their stability have been a matter of intensive research due to the importance of these inequalities in applications to geometric analysis, partial differential equations, convex geometry, mathematical physics and problems in spectral theory and stability of matter, etc. While many mathematicians have made contributions in this direction, it is impossible to review all the results. We will begin with presenting a brief history of the main results on these problems only closely related to the questions under consideration in this paper.

\medskip

\subsection{Functional inequalities on the sphere $\mathbb{S}^n$}\
\vskip0.3cm
The classical sharp Sobolev inequality in $\mathbb{R}^n$  for $0<s<\frac n2$  states
\begin{equation}
	\label{eq-sob}
	\left\|(-\Delta)^{s/2} U \right\|_2^2 \geq \mathcal S_{s,n} \| U\|_{\frac{2n}{n-2s}}^2
	\qquad\text{for all}\ U\in \dot H^s(\mathbb{R}^n)
\end{equation}
with
\begin{equation}
	\label{eq:sobconst}
	\mathcal S_{s,n} = (4\pi)^s \ \frac{\Gamma(\frac{n+2s}{2})}{\Gamma(\frac{n-2s}{2})} \left( \frac{\Gamma(\frac n2)}{\Gamma(n)} \right)^{2s/n}
	= \frac{\Gamma(\frac{n+2s}{2})}{\Gamma(\frac{n-2s}{2})} \ |\mathbb{S}^n|^{2s/n} \,,
\end{equation}
where $\dot H^s(\mathbb{R}^n)$ denotes the $s$-order homogenous Sobolev space in $\mathbb{R}^n$: the completion of $C_{c}^{\infty}(\mathbb{R}^n)$ under the norm$\big(\int_{\mathbb{R}^n}|(-\Delta)^{\frac{s}{2}}U|^2dx\big)^{\frac{1}{2}}$. The sharp constant of inequality (\ref{eq-sob}) has been computed first by Rosen \cite{Ro} in the case $s=2$, $n=3$ and then independently by Aubin \cite{Au} and Talenti \cite{Ta} in the case $s=2$ and general $n$. For general $s$, the sharp inequality was proved by Lieb in \cite{Li} as an equivalent reformulation of sharp Hardy-Littlewood-Sobolev inequality in the case of conformal index. Furthermore, he also showed that the equality of Sobolev inequality \eqref{eq-sob} holds if and only if
$$U\in M_s:=\left\{cV(\frac{\cdot-x_0}{a}): c\in \mathbb{R},\  x_0\in \mathbb{R}^n,\  a>0 \right\},$$
where $V(x)=(1+|x|^2)^{-\frac{n-2s}{2}}$.

\medskip

By the stereographic projection, we know that $\mathbb{R}^n$ (or rather $\mathbb{R}^n\cup\{\infty\}$) and $\mathbb{S}^n$ ($\subset \mathbb{R}^{n+1}$) are conformally equivalent. Thus, there exists an equivalent version of \eqref{eq-sob} on $\mathbb{S}^n$. This form was found explicitly by Beckner in \cite[Eq.~(19)]{Be1993}, namely,
\begin{eqnarray}
	\label{eq:sobsphere}
	\left\| A_{2s}^{1/2} u \right\|_2^2 \geq \mathcal S_{s,n} \|u\|_{\frac{2n}{n-2s}}^2
	\qquad\text{for all}\ u\in H^s(\mathbb{S}^n)
\end{eqnarray}
with
\begin{align}
	\label{eq:opas}
	A_{2s} = \frac{\Gamma(B+\tfrac12 + s)}{\Gamma(B+\tfrac12 - s)}
	\qquad\text{and}\qquad
	B = \sqrt{-\Delta_{\mathbb{S}^n} + \tfrac{(n-1)^2}{4}}, \
\end{align}
where $-\Delta_{\mathbb{S}^n}$ denotes the Laplace-Beltrami operator on the sphere $\mathbb{S}^n$ and the operator $B$ and $A_{2s}$ acting on spherical harmonics $Y_{l,m}$ of degree $l$   satisfy
 $$B Y_{l,m}=(l+\frac{n-1}{2})Y_{l,m},~~~A_{2s}Y_{l,m}=\frac{\Gamma(l+n/2+s)}{\Gamma(l+n/2-s)}Y_{l,m}.$$
 (We refer the reader to the books \cite{M} and \cite{SteinWeiss} for detailed exposition of spherical harmonics, and also Subsection 2.1 for a brief account.)
In fact, the operators $A_{2s}$ is a $2s$-order conformally invariant differential operator and can be written as  $(-\Delta_{\mathbb{S}^n})^{s}+lower\ order\  terms$. For the integer $s$, they are related to the GJMS operators in conformal geometry \cite{FG, GrJeMaSp}.
Beckner also proved that the equality holds in (\ref{eq:sobsphere}) if and only if
\begin{eqnarray}\label{opt set}
u\in M_\ast:=\left\{c(1-\xi\cdot \omega)^{\frac{2s-n}{2}}:\  \xi\in B^{n+1}, \ c\in \mathbb{R} \right\}.
\end{eqnarray}

Inequality (\ref{eq:sobsphere}) becomes an equality as $s\rightarrow 0$. Differentiating at $s=0$ and using the Funk-Hecke formula, Beckner \cite{Be1992,Be1997} proved the following invariant logarithmic Sobolev inequality on
$\mathbb{S}^n$.
\vskip0.5cm
\textbf{Theorem A. }
\textit{Assume} $u\in L^2(\mathbb{S}^n)$. \textit{Then}
\begin{align}\label{L-S ine}
\iint_{\mathbb{S}^n \times \mathbb{S}^n} \frac{|u(\omega) - u(\eta)|^2}{|\omega - \eta|^n} \, d\omega d\eta  \geq C_n \int_{\mathbb{S}^n} |u(\omega)|^2 \ln \frac{|u(\omega)|^2 |\mathbb{S}^n|}{\|u\|_2^2} d\omega
\end{align}
 \textit{with sharp constant}
 \begin{equation}
\label{eq:becknerconst}
C_n = \frac{4}{n} \frac{\pi^{n/2}}{\Gamma(n/2)} \,.
\end{equation}
\textit{The equality holds if and only if}
  \begin{eqnarray}
\label{eq:beckneropt}
u(\omega) = c \left( 1-\xi\cdot\omega\right)^{-n/2}
\end{eqnarray}
 \textit{for some} $\xi\in B^{n+1}:=\{\xi\in \mathbb{R}^{n+1}, |\xi|<1\}$ \textit{and some} $c\in \mathbb{R}$.
\vskip0.5cm

 Note that as $s\nearrow \frac n2$, the integrability exponent $\frac{2n}{n-2s}$ in \eqref{eq-sob} and \eqref{eq:sobsphere} tends to $+\infty$. In \cite{On}, Onofri \cite{On} derived an endpoint inequality on the 2 dimensional sphere, which is known as the Moser-Onofri inequality. He proved that
$$\ln \int_{\mathbb{S}^2} \!\!\!\!\!\!\!\!\!\!\!\!\; {}-{} \,\,\, e^f d \omega\leq  \int_{\mathbb{S}^2} \!\!\!\!\!\!\!\!\!\!\!\!\; {}-{} \,\,\, f d \omega+\frac{1}{4}
\int_{\mathbb{S}^2} \!\!\!\!\!\!\!\!\!\!\!\!\; {}-{} \,\,\, |\nabla f|^2 d \omega, ~~~\text{for}~~~f\in H^{1}(\mathbb{S}^2). $$
We also refer the reader to an independent proof  by Hong \cite{Hong}.
In 1993, using the endpoint differentiation argument, spherical harmonics techniques and Lieb's \cite{Li} Hardy-Littlewood-Sobolev inequality on the sphere, Beckner established in \cite{Be1993} the Moser-Onofri inequality on the higher dimensional sphere for all $n\geq 1$ and $s=n/2$.

\vskip0.5cm
\textbf{Theorem B. }
\textit{Let} $f\in H^{\frac{n}{2}}(\mathbb{S}^n)$. \textit{Then}
\begin{align}\label{M-T ine}
\ln \int_{\mathbb{S}^n} \!\!\!\!\!\!\!\!\!\!\!\!\; {}-{} \,\,\, e^f d \omega\leq  \int_{\mathbb{S}^n} \!\!\!\!\!\!\!\!\!\!\!\!\; {}-{} \,\,\, f d \omega+\frac{1}{2n!}
\int_{\mathbb{S}^n} \!\!\!\!\!\!\!\!\!\!\!\!\; {}-{} \,\,\, f(A_{n}f) d \omega.
\end{align}
 \textit{This inequality is conformal invariant under the transformation}
 $$f(\omega)\rightarrow f(\Phi(\omega))+\ln J_{\Phi}(\omega),$$
  \textit{where} $\Phi$ \textit{is a conformal map on} $\mathbb{S}^n$. \textit{The equality holds if and only if}
  $$f(\omega)=-n\ln(1-\xi\cdot\omega)+C$$
  \textit{for some} $\xi\in B^{n+1}$ \textit{and some} $C\in \mathbb{R}$.
\vskip0.5cm

Here we need to point out that $A_{n}$ acting on spherical harmonics $Y_{l,m}$ satisfies
 $$A_{n}Y_{l,m}=\frac{\Gamma(l+n)}{\Gamma(l)}Y_{l,m}~~\text{for}~~l\neq 0$$
  and $A_{n}Y_{0,1}=0$.
For the range $s>n/2$,  a "reverse" type of Sobolev inequality was considered. We refer the interested readers to the recent paper \cite{FrKoTa1} and references therein.

\subsection{Stability of Geometric and Functional inequalities}\

\vskip0.3cm
Assume that the following functional inequality holds:
 \begin{eqnarray}\label{func ine}
 \mathcal E(x)\leq \mathcal F(x)~~\text{for all} ~~x\in X,
  \end{eqnarray}
 where $\mathcal E$ and $\mathcal F$ are two functionals defined on the linear space $X$. The functional inequality (\ref{func ine}) is said to be optimal if $\mathcal E(x)\leq \mathcal F(x)$ for all $x\in X$ and for any $\lambda<1$, there exists some $x\in X$ such that $\mathcal E(x)\geq \lambda\mathcal F(x)$. Let
$$X_0=\{x\in X: \mathcal E(x)=\mathcal F(x)<\infty\}$$
be the set of optimizers, we say that the functional inequality (\ref{func ine}) is sharp if $X_0\neq \emptyset$. It should be noted that a
sharp functional inequality is necessarily optimal, but not vice-versa.
\vskip 0.1cm

Once a sharp functional inequality (\ref{func ine}) in a normed linear space $X$ was established, we are interested in the stability of the inequality, which means finding
a suitable metric $d$ in $X$ (not necessarily being the metric induced by the norm) and a function $\Phi: [0,\infty)\rightarrow[0,\infty)$ such that
\begin{eqnarray}\label{global-sta}
\mathcal F(x)-\mathcal E(x)\geq \Phi(d(x,X_0)) ~~\text{for all}~~x\in X.
\end{eqnarray}
In this paper, we will call the inequality (\ref{global-sta}) the global stable inequality compared with the following local stable inequality:
\begin{eqnarray}\label{local-sta}
\mathcal F(x)-\mathcal E(x)\geq \Phi(d(x,X_0))+o(d(x,X_0)).
\end{eqnarray}
Obviously, the global stability implies the local stability but not vice-versa. However, we will reveal the fact that the local stability and global stability of functional inequalities are actually equivalent under some mild conditions.
\medskip

In fact, assume that the sharp functional inequality
\begin{eqnarray*}
\mathcal E(x) \leq  \mathcal F(x) ~~\text{for all}~~x\in X
\end{eqnarray*}
holds and  the functional $\mathcal E(x)$ is $p$-homogeneous, which means $ \mathcal E(\lambda x)=\lambda^p \mathcal E(x)$ for all $\lambda>0$ and
the functional $\mathcal F(x)$ is also $p$-homogeneous. Let $x_0$ be an element of the optimizer set $X_0$. Obviously,
$\lambda x_0\in X_0$ for any $\lambda>0$. We also assume that the distance $d$ satisfies $d(x,X_0)\not\equiv 0$ and
$$d(\lambda x,\lambda y)=\lambda d(x,y), \ \forall\  \lambda>0,\  x,y \in X.$$
Then we have the following theorem.
\vskip0.5cm
\begin{theorem}\label{l-sta imply g-sta}
Let $\mathcal E, \mathcal F, X_0, d$ satisfy the above assumptions. Then\\
(i) If $q\neq p$, there is no global stable inequality of the following form
$$\mathcal F(x)-\mathcal E(x)\geq c d^{q}(x,X_0)~~\text{for all}~~x\in X.$$
If $q<p$, for $x\in X$ satisfy $0<d(x,X_0)$ small enough, there exists no local stable inequality of the following form
$$\mathcal F(x)-\mathcal E(x)\geq c d^{q}(x,X_0)+o(d^q(x,X_0)).$$
(ii) The local stable inequality
\begin{eqnarray}\label{l-sta}
\mathcal F(x)-\mathcal E(x)\geq c d^{p}(x,X_0)+o(d^p(x,X_0))
\end{eqnarray}
implies the global stable inequality
\begin{eqnarray}\label{g-sta}
\mathcal F(x)-\mathcal E(x)\geq c d^{p}(x,X_0).
\end{eqnarray}
Moreover, if $c_0$ is the optimal constant of the local stable inequality (\ref{l-sta}) then it is also the optimal constant of the global stable inequality (\ref{g-sta}).
\end{theorem}

\vskip 0.1cm
\subsection{Stability of Sobolev inequalities}

Research on the stability of Sobolev inequality started from the work by Brezis and Lieb. In \cite{BrLi} they asked if the following refined first order Sobolev inequality ($s=1$ in (\ref{eq-sob}))
holds for some distance function $d$:
$$\left\|(-\Delta)^{1/2} U \right\|_2^2 - \mathcal S_{1,n} \| U\|_{\frac{2n}{n-2}}^2\geq c d^{2}(U, M_1).$$
This question was answered affirmatively in the case $s = 2$ in a pioneering work by Bianchi and Egnell \cite{BiEg}, in the case $s=4$ by the second author and Wei \cite{LuWe} and in the case of any positive even integer $s<n/2$ by
Bartsch, Weth and Willem \cite{BaWeWi}. In 2013, Chen, Frank and Weth \cite{ChFrWe} established the stability of Sobolev inequality for all $0<s<n/2$.

\vskip0.5cm
\textbf{Theorem C. } \cite{ChFrWe}
\textit{There exists a positive constant} $c$ \textit{depending only
on the dimension} $n$  \textit{and} $s\in (0,n/2)$ \textit{such that}
\begin{align}\label{sta sob}
\left\|(-\Delta)^{s/2} U \right\|_2^2 - \mathcal S_{s,n} \| U\|_{\frac{2n}{n-2s}}^2\geq c d^{2}(U, M_s),
\end{align}
 \textit{for all} $U\in \dot H^s(\mathbb{R}^n)$,
  \textit{where} $d(U,M_s)=\min\{\|(-\Delta)^{s/2}(U-\phi)\|_{L^2}:\phi \in M_s\}$.
\vskip0.5cm
In fact, via the stereographic projection, they first proved the following equivalent local stable Sobolev inequality on the sphere
\begin{align}\label{loc sta sob}
\left\| A_{2s}^{1/2} u \right\|_2^2 - \mathcal S_{s,n} \|u\|_{\frac{2n}{n-2s}}^2 \geq \frac{4s}{n+2s+2}d^{2}(u,M_\ast)+o(d^{2}(u,M_\ast)),
\end{align}
where $M_\ast$ is the optimizer set and $d(u,M_\ast)=\min\{\| A_{2s}^{1/2} (u-v) \|_2^2: v\in M_\ast\}$.
\vskip0.1cm

To obtain the global stability, the concentration compactness type argument plays an important role in  showing that
$\left\|(-\Delta)^{s/2} U_n \right\|_2^2 - \mathcal S_{s,n} \| U_n\|_{\frac{2n}{n-2s}}^2\rightarrow 0$ implies $d(U_n, M_s)\rightarrow 0$ for
$U_n \in H^s(\mathbb{R}^n)$. It is notable that they pointed out that the constant $\frac{4s}{n+2s+2}$ in (\ref{loc sta sob}) is sharp (at the end
of their first section). Because of using the concentration-compactness argument, they could not give any lower bound information about the sharp constant of stable fractional Sobolev inequality. In fact, to our knowledge,  the sharp stability for Sobolev inequalities with the best constants still remain open (see \cite{Carlen,CFMP,BiEg}).    We also note that Dolbeault et al. \cite{DEFL} recently gave the lower bound estimate for sharp constant of stable global Sobolev inequalities.
\medskip

Though the local stability of the Sobolev inequality has been established in Theorem C, this local stability does not satisfy the local stability assumption required in Theorem \ref{l-sta imply g-sta}. Therefore, we cannot conclude that the global stability of the Sobolev inequalities with the same best constants can follow from the local stability stated in Theorem C.  In fact,
K\"{o}nig has recently showed that the best constant for the global stability of the Sobolev inequality must be strictly smaller than $\frac{4s}{n+2s+2}$. (see \cite{Ko}).

\medskip

All the above stability of Sobolev inequalities were established in a Hilbert space. In fact, for some other type functional inequalities such as
Gagliardo-Nirenberg, weighted-Hardy-Sobolev and Caffarelli-Kohn-Nirenberg  inequalities, their stability can also been done in the framework of Hilbert space, one can refer to
\cite{CaF,CFLL, Dong, WW} for details. For the study of stability of Sobolev inequalities in non-Hilbert space, we refer the interested readers to the papers \cite{Carlen}, \cite{CFMP}, \cite{FiNe}, \cite{FiZh} and references therein.

\subsection{Stability of the Log-Sobolev and Moser-Onofri inequalities on the sphere  $\mathbb{S}^n$}

\medskip

In this paper, our main goals are to establish the stability of the endpoint conformally invariant inequalities on the sphere $\mathbb{S}^n$, namely the log-Sobolev inequality (\ref{L-S ine}) and
the Moser-Onofri inequality (\ref{M-T ine}). We will also explore the optimal constants of these stable inequalities.
\medskip

Let
$$M=\left\{v_{c,\xi}(\omega)=c\left(1-\xi\cdot w\right)^{-n/2}:\, ~\xi\in B^{n+1}, \ c\in \mathbb{R}\right\}$$
be the optimizer set of the log-Sobolev inequality on the sphere $\mathbb{S}^n$. Denote the log-Sobolev functional by
$$LS(u)=\iint_{\mathbb{S}^n \times \mathbb{S}^n} \frac{|u(\omega) - u(\eta)|^2}{|\omega - \eta|^n} \, d\omega d\eta  - C_n \int_{\mathbb{S}^n} |u(\omega)|^2 \ln \frac{|u(\omega)|^2 |\mathbb{S}^n|}{\|u\|_2^2} d\omega.$$
Our first goal in the paper is to explore the stability of the log-Sobolev inequality on the sphere $\mathbb{S}^n$. First of all, we need to find a suitable distance function. It is natural to adopt the distance induced by the double integral in the
inequality, namely to establish a type of stability inequality in the form of
$$LS(u)\geq c d^2_{0}(u,M),$$
for some $c>0$, where
$$d^2_{0}(u,v)=\iint\limits_{\mathbb{S}^n\times \mathbb{S}^n}\frac{|u(\omega)-v(\omega)-u(\eta)+v(\eta)|^2}{|\omega-\eta|^n}d \omega d \eta.$$
Nevertheless, we can't expect that this stability inequality holds (see Theorem \ref{Main Theorem 2}).
\vskip0.1cm

Instead, we will choose the $L^2$ distance since the log-Sobolev inequality is $L^2$ conformally invariant.
Denote $$\mathcal{D_L}=\left\{v\in L^2(\mathbb{S}^n):\iint\limits_{\mathbb{S}^n\times \mathbb{S}^n}\frac{|v(\omega)-v(\eta)|^2}{|\omega-\eta|^n}d \omega d \eta<\infty\right\}.$$
For any functions $u,v\in \mathcal{D}_{L}$, we define the inner product
$$\left\langle u, v \right\rangle=\int_{\mathbb{S}^n}u(\omega)\overline{v(\omega)}d\omega.$$
We also define $$d^2(u,M)=\inf\{\|u-\phi\|_{L^2}^2:\phi\in M\}=\inf_{(c,\xi)\in \mathbb{R}\times B^{n+1}}\left\langle u-v_{c,\xi}, u-v_{c,\xi} \right\rangle,$$
where $B^{n+1}=\{x\in \mathbb{R}^{n+1}: \,|x|<1\}.$

\medskip

In this paper we first establish the local stability of the log-Sobolev inequality~(\ref{L-S ine}) in the setting of $L^2$ distance.
\vskip0.5cm
\begin{theorem}\label{Main Theorem 1}
Let $u\in \mathcal{D}_{L}$ with $\|u\|_{L^2}=1$.  Then
$$LS(u) \geq \frac{8\pi^{n/2}}{\Gamma(n/2)(n+2)} d^2(u,M)+o(d^2(u,M)),$$
where $o(d^2(u,M))$ is only dependent on $d(u, M)$ but independent of $u$. Moreover, the constant $\frac{8\pi^{n/2}}{\Gamma(n/2)(n+2)}$ is sharp.

\end{theorem}
\vskip0.5cm

Moreover, motivated by the recent work of the global stability of the Sobolev inequalities by K\"{o}nig \cite{Ko}, we can also establish the following result.
\begin{theorem}\label{thm2}
Assume that there exists $C>0$ independent of $u$ such that $$LS(u) \geq C d^2(u,M)$$
for any $u\in \mathcal{D}_{L}$. Then $C$ is strictly smaller than the optimal constant $\frac{8\pi^{n/2}}{\Gamma(n/2)(n+2)}$ of locally stable log-Sobolev inequality.
\end{theorem}

Next, we will explain why we can't expect the type of log-Sobolev stability inequality with the distance $d_{0}$ to hold.
\vskip0.5cm
\begin{theorem}\label{Main Theorem 2}
 There exists a sequence of functions $\{u_m\}\subset \mathcal{D}_{L}$ such that
$$\frac{LS(u_m)}{d_{0}^2(u_m,M)}\rightarrow 0.$$
\end{theorem}
\vskip0.5cm

Our second goal in this paper is to try to set up the stability of the sharp Moser-Onofri inequality on the sphere $\mathbb{S}^n$. The main difficulty in solving this problem lies in finding a suitable distance function between a function
$u$ and the set of optimizer. Although the authors in \cite{DEJ} employed the entropy and flow method to give the improved Moser-Onofri inequality,  the remainder term is quite complicated and cannot be regarded  as some distance  function. In order to overcome this difficulty, we will first change the form of the Moser-Onofri inequality.
\medskip

Let $u=e^{f/2}$. Then the sharp inequality~(\ref{M-T ine}) becomes the following equivalent inequality:
\begin{align}\label{M-T ine 1}
\ln \int_{\mathbb{S}^n} \!\!\!\!\!\!\!\!\!\!\!\!\; {}-{} \,\,\, u^2 d \omega\leq  \int_{\mathbb{S}^n} \!\!\!\!\!\!\!\!\!\!\!\!\; {}-{} \,\,\, \ln u^2 d \omega+\frac{1}{2n!}
\int_{\mathbb{S}^n} \!\!\!\!\!\!\!\!\!\!\!\!\; {}-{} \,\,\, \ln u^2\big(A_{n}(\ln u^2)\big) d \omega~~\text{for}~~u\in \mathcal{Q}_M,
\end{align}
where
$$\mathcal{Q}_{M}=\{u\in L^2(\mathbb{S}^n): \ln u^2\in H^{\frac{n}{2}}(\mathbb{S}^n)\}.$$
Then the inequality~(\ref{M-T ine 1}) is conformally invariant under the transformation
$$u(\omega)\rightarrow  J^{1/2}_{\Phi}(\omega)u(\Phi(\omega)),$$
where $\Phi$ is a conformal map on $\mathbb{S}^n$ and the equality holds if and only if $u$ belongs to
$$M=\left\{v_{c,\xi}=c\left(1-\xi\cdot w\right)^{-n/2}:\, ~\xi\in B^{n+1},\  c\in \mathbb{R}\right\}.$$
Once we have observed the $L^2$ conformal invariance of the inequality~(\ref{M-T ine 1}), it is natural for us to adopt the $L^2$ distance to study the stability of the new form of Moser-Onofri inequality on the sphere $\mathbb{S}^n$. Now we define the Moser-Onofri functional $MO(u)$ by
$$MO(u)= \frac{1}{2n!}\int_{\mathbb{S}^n} \!\!\!\!\!\!\!\!\!\!\!\!\; {}-{} \,\,\, \ln u^2\big(A_{n}(\ln u^2)\big) d \omega+\int_{\mathbb{S}^n} \!\!\!\!\!\!\!\!\!\!\!\!\; {}-{} \,\,\, \ln u^2 d \omega-\ln \int_{\mathbb{S}^n} \!\!\!\!\!\!\!\!\!\!\!\!\; {}-{} \,\,\, u^2 d \omega.$$
 We will establish the existence of the local stability for the Moser-Onofri inequality and the nonexistence of the global stability for the Moser-Onofri inequality on the sphere $\mathbb{S}^n$.

\vskip0.5cm
\begin{theorem}\label{Main Theorem 3}
For $u\in \mathcal{Q}_{M}$ with $\|u\|_{L^2}=1$, we have
 $$MO(u) \geq \frac{\Gamma(n+1)}{2^{n-1}\pi^{n/2}\Gamma(n/2)}d^2(u,M)+o(d^2(u,M)),$$
where $o(d^2(u,M))$ is only dependent on $d(u, M)$ but independent of $u$.
Furthermore, there exists a sequence of functions  $\{u_m\}\subset \mathcal{Q}_{M}$ such that
$$\frac{MO(u_m)}{d^2(u_m,M)}\rightarrow 0$$
which means that the global stability for the Moser-Onofri inequality on the sphere does not hold.
\end{theorem}
\vskip0.5cm

This paper is organized as follows. In Section 2, we will prove
Theorem \ref{l-sta imply g-sta} which reveals the relation between the local and global stability of functional inequalities in a fairly general framework. We also characterize the tangent space of the manifold of the optimizers of the log-Sobolev inequality on the sphere $\mathbb{S}^n$.
Section 3 is devoted to giving the sharp stability of the log-Sobolev inequality with $L^2$ distance (Theorem \ref{Main Theorem 1}) and that the log-Sobolev inequality is unstable with respect to the distance $d_0$ (Theorem \ref{Main Theorem 2}). In Section 4, we will prove the sharp local stability for the Moser-Onofri inequality and the nonexistence of the global stability for the Moser-Onofri inequality on the sphere $\mathbb{S}^n$ (Theorem \ref{Main Theorem 3}).

{\bf Acknowledgement.} We wish to thank Rupert Frank for pointing out to us that our Theorem \ref{l-sta imply g-sta} in the general framework cannot imply the global stability of the Sobolev inequalities with the same best constants as in the local stability of the Sobolev inequalities as shown by the counterexample of T. K\"{o}nig.
We also thank Tobias K\"{o}nig for sending us his preprint \cite{Ko}.

\section{Preliminary}

\subsection{Spherical harmonics}\
\vskip0.3cm
In this subsection, we will recall some facts about spherical harmonics and a detailed discussion on spherical harmonics can be found in, e.g.,  \cite{M} and \cite{SteinWeiss}. In brief, spherical harmonics are restrictions to the
 unit sphere $\mathbb{S}^n$ of polynomial $Y(x)$, which satisfy $\Delta_{x}Y=0$, where $\Delta_{x}$ is the Laplacian operator in $\mathbb{R}^{n+1}$. The space of all spherical harmonics of degree $l$ on $\mathbb{S}^n$, denoted by $\mathcal{H}_{l}^n$, has an orthonormal basis
$$\{Y_{l,m}:m=1,\cdots,N(n,l),\ l=1, 2, \cdot\cdot\cdot \},$$
where
$$N(n,0)=1~~ \text{and}~~ N(n,l)=\frac{(2l+n-1)(\Gamma(l+n-1))}{\Gamma(l+1)\Gamma(n)}.$$
Every $u\in{L^2(\mathbb{S}^n)}$ can be expanded in terms of spherical harmonics
$$u=\sum_{l=0}^{+\infty}\sum_{m=1}^{N(n,l)}u_{l,m}Y_{l,m}, ~~\text{where}~~u_{l,m}=\int_{\mathbb{S}^n}uY_{l,m}d\omega,$$
and by Parseval's identity,
$$\|u\|^2_{L^2(\mathbb{S}^n)}=\int_{\mathbb{S}^n}|u|^2d\omega=\sum_{l=0}^{+\infty}\sum_{m=1}^{N(n,l)}|u_{l,m}|^2.$$
And the Sobolev space $H^{s}(\mathbb{S}^n)$ on the sphere is defined by $$H^{s}(\mathbb{S}^n):=\{u\in{L^2(\mathbb{S}^n):\|u\|^2_{H^s}=\sum_{l=0}^{+\infty}(1+l(l+n-1))^s}\sum_{m=1}^{N(n,l)}|u_{l,m}|^2<+\infty\}.$$
\subsection{Local stability implies global stability: Proof of Theorem \ref{l-sta imply g-sta}}\
\vskip0.3cm
In this subsection, we will reveal the relation between the local and global stability of functional inequalities. Namely, we shall gave the proof
of Theorem \ref{l-sta imply g-sta}
\medskip

\begin{proof}
When $q\neq p$, we will first prove the conclusion (i) by contradiction. Assume that the global stability inequality
$$\mathcal F(x)-\mathcal E(x)\geq c d^{q}(x,X_0) ~~\text{for all}~~x\in X$$
holds. Fix $x\in X$ such that $d(x,X_0)\neq 0$ and $\mathcal F(x)<\infty$, and define $x_\lambda=\lambda x$ for $\lambda>0$. Then it is clear that $d(x_\lambda,X_0)=\lambda d(x,X_0)$. Thus
$$\frac{\mathcal F(x_\lambda)-\mathcal E(x_\lambda)}{d^{q}(x_\lambda,X_0)}=\lambda^{p-q}\frac{\mathcal F(x)-\mathcal E(x)}{d^{q}(x,X_0)},$$
which will go to $0$ as $\lambda\rightarrow 0$ when $p>q$ and as $\lambda\rightarrow \infty$ as $p<q$. That is a contradiction.
Thus, part (i) of Theorem \ref{l-sta imply g-sta} is proved. Using the same method, we can prove that the local stable inequality does not hold if $q<p$.
\vskip0.1cm

Next, we will  prove part (ii) of  Theorem \ref{l-sta imply g-sta}. Namely,  if $p=q$, then the local stability implies the global stability. Indeed, for any fixed $x\in X$ such that $d(x,X_0)\neq 0$, choose $x_m=\frac{x}{m}$, a simple calculation gives
$d(x_m,X_0)=\frac{1}{m}d(x,X_0)\rightarrow 0$. Applying the local stable inequality, we obtain
\begin{eqnarray*}
\mathcal F(x_m)-\mathcal E(x_m)\geq c d^{p}(x_m,X_0)+o(d^p(x_m,X_0)).
\end{eqnarray*}
Since $\mathcal E(x)$ and $\mathcal F(x)$ are $p$-homogeneous, then
$$\frac{\mathcal F(x)-\mathcal E(x)}{d^{p}(x,X_0)}=\frac{\mathcal F(x_m)-\mathcal E(x_m)}{d^{p}(x_m,X_0)}\geq c + o(1),$$
which implies the global stability inequality \eqref{g-sta}.
\vskip0.1cm

Finally, we will show that the sharp constant of global stable inequality
is consistent with that of the local stable inequality. Assume that $c_0$ and $c_1$ are the optimal constants of the local and global stable inequalities respectively. Obviously $c_1\leq c_0$. If $c_1<c_0$, then there exist a
sequence $\{x_i\}$ such that
$$\frac{\mathcal F(x_i)-\mathcal E(x_i)}{d^{p}(x_i,X_0)}\rightarrow c_1<c_0.$$
Without loss of generality, we may assume $d(x_i,X_0)$ is bounded and has a positive lower bound. Otherwise, we can choose $x_{i}'=\frac{x_i}{d(x_i, X_0)}$. Consider the sequence $\{\frac{x_i}{i}\}$, then $d(x_i/i,X_0)\rightarrow 0$ and
$$\frac{\mathcal F(x_i/i)-\mathcal E(x_i/i)}{d^{p}(x_i/i,X_0)}\rightarrow c_1<c_0,$$
which is a contradiction with the fact that $c_0$ is the optimal constant of the local stable inequality\eqref{l-sta}. Therefore, we must have $c_1=c_0$, which completes the proof.
\end{proof}

\subsection{The tangent space}\
\vskip0.3cm

In this subsection, we will characterize the tangent space of the optimizer set of log-Sobolev inequality (\ref{L-S ine}). and Moser-Onofri inequality (\ref{M-T ine 1}).
Set
$$M=\left\{v_{c,\xi}=c\left(\frac{\sqrt{1-|\xi|^2}}{1-\xi\cdot w}\right)^{n/2}:\, ~\xi\in \mathbb{R}^{n+1}, c\in \mathbb{R}, |\xi|<1\right\}.$$
Since $M$ can be viewed as
 an $n+2$ dimensional smooth manifold embedded in $L^2(\mathbb{S}^n)$ via the mapping
 $$\mathbb{R}\times B^{n+1}\rightarrow L^2(\mathbb{S}^n),~~~(c,\xi)\rightarrow v_{c,\xi},$$
then the tangent space of $M$ at $v_{c_0,\xi_0}$ is
$$TM_{v_{c_0,\xi_0}}=\text{span}\{v_{1,\xi_0},\partial_\xi v_{c,\xi}|_{(c_0, \xi_0)} \}.$$
For any element $v_{c_0,\xi_0}$ in $M$, by the classification of conformal maps of $\mathbb{S}^n$, there exists a conformal transformation $\Phi$ (depending on $\xi_0$, for convenience we omit $\xi_0$) such that
$$v_{c_0,\xi_0}(\omega)=c_0J_{\Phi}^{1/2}=c_0\left(\frac{\sqrt{1-|\xi_0|^2}}{1-\xi_0\cdot w}\right)^{n/2}.$$
We will characterize $TM_{v_{c_0,\xi_0}}$ by spherical harmonics of degrees $0$ and $1$. But at first we need the following lemma, which was proved by
Frank, K\"onig and the third author~\cite{FrKoTa2} in order to establish the conformal invariance of the Euler-Lagrange equation related to the log-Sobolev inequality (\ref{L-S ine}).
 For any $v\in L^2(\mathbb{S}^n)$, let us denote $v_{\Phi^{-1}}$ by
$$v_{\Phi^{-1}}=J_{\Phi^{-1}}^{1/2}v\circ\Phi^{-1}.$$
We also recall that
 $$\mathcal{D_L}=\left\{v\in L^2(\mathbb{S}^n):\iint\limits_{\mathbb{S}^n\times \mathbb{S}^n}\frac{|v(\omega)-v(\eta)|^2}{|\omega-\eta|^n}d \omega d \eta<\infty\right\}.$$
\begin{lemma}
\label{lemma conf inv}
Let $u, v \in \mathcal D_{L}$ and $\Phi$ be a conformal map on $\mathbb{S}^n$. Then $u_\Phi, v_\Phi \in \mathcal D_{L}$ and
\begin{equation*}
H (u_\Phi)
= (H u)_\Phi
+ C_n u_\Phi \ln J_\Phi^\frac 12 \,,
\end{equation*}
where$$H(u)(\omega)=P.V.\int_{\mathbb{S}^n}\frac{u(\omega)-u(\eta)}{|\omega-\eta|^n}d\eta$$
is the principal value of the integral.
\end{lemma}
Now we are in the position to characterize the tangent space $TM_{v_{c_0,\xi_0}}$.
\begin{lemma}\label{cha of tan spa}
$$TM_{v_{c_0,\xi_0}}=\text{span}\{J_{\Phi}^{1/2}Y_{0},J_{\Phi}^{1/2}Y_{1,i}\circ \Phi,~~i=1,\cdots,n+1\}.$$
\end{lemma}
\vskip0.5cm
\begin{proof}
In order to prove this lemma, we will take advantage of the fact that $v_{c_0,\xi_0}$ is the element of the optimizer set of log-Sobolev inequality (\ref{L-S ine}).
Since $v_{c,\xi}$ is the extremal function of the
log-Sobolev inequality, then it must satisfy the following Euler-Lagrange equation
$$H(u)=\frac{C_n}{2} u \ln \frac{u^2|\mathbb{S}^n|}{\|u\|_2^2},$$
where $C_n = \frac{4}{n} \frac{\pi^{n/2}}{\Gamma(n/2)}$ is the sharp constant of log-Sobolev inequality (\ref{L-S ine}). Differentiating at $(c_0,\xi_0)$, we know that $v_{1,\xi_0}$ and $\partial_{\xi_{i}} v_{c,\xi}|_{(c_0, \xi_0)}$ for $i=1, 2, \cdot\cdot\cdot, n+1$ satisfying the following equation
\begin{align}\label{equation}
H(u)=\frac{C_n}{2}u\ln\frac{v^2_{c_0,\xi_0}|\mathbb{S}^n|}{\|v_{c_0,\xi_0}\|_2^2}+C_n u-C_n\frac{\int_{\mathbb{S}^n}v_{c_0,\xi_0}u d\omega}{\|v_{c_0,\xi_0}\|_2^2}v_{c_0,\xi_0},
\end{align}
which implies that the tangent space belongs to the solution space of the equation (\ref{equation}).
\vskip0.1cm

Next, we will show that in fact, the tangent space is equal to the solution space of the equation (\ref{equation}) by the dimensions of the two spaces. We already know that $v_{c_0,\xi_0}=c_0J_{\Phi}^{1/2}$ for some conformal transformation $\Phi$ on $\mathbb{S}^n$. For any solution $u$ of the equation (\ref{equation}), let $u_{\Phi^{-1}}=J_{\Phi^{-1}}^{1/2}u\circ\Phi^{-1}$. Then by Lemma \ref{lemma conf inv} we know that
\begin{equation}
\label{conf transf H}
H(u_{\Phi^{-1}})=(H(u))_{\Phi^{-1}}+C_n u_{\Phi^{-1}}\ln J_{\Phi^{-1}}^{1/2}.
\end{equation}
Using (\ref{equation}) and $v_{c_0,\xi_0}=c_0J_{\Phi}^{1/2}$ we can obtain
\begin{align}\label{conf transf H(u)}\nonumber
(H(u))_{\Phi^{-1}}&=\frac{C_n}{2}u_{\Phi^{-1}}\ln (J_{\Phi}\circ{\Phi}^{-1})+C_n u_{\Phi^{-1}}-C_n\frac{\int_{\mathbb{S}^n}J_{\Phi}^{1/2}u d\omega}{|\mathbb{S}^n|}J_{\Phi^{-1}}(J_{\Phi}\circ{\Phi}^{-1})\\
& =-\frac{C_n}{2}u_{\Phi^{-1}} \ln J_{\Phi^{-1}}+C_n u_{\Phi^{-1}}-C_n\int_{\mathbb{S}^n}u_{\Phi^{-1}}d\omega\frac{1}{|\mathbb{S}^n|}.
\end{align}
Then it follows from (\ref{conf transf H}) and ({\ref{conf transf H(u)}}) that
\begin{align*}
H(u_{\Phi^{-1}})=C_n \left[u_{\Phi^{-1}}-\int_{\mathbb{S}^n}u_{\Phi^{-1}}d\omega\frac{1}{|\mathbb{S}^n|}\right].
\end{align*}
Then by the eigenvalue problem (\ref{eigenvalue problem}) which we will solve in the next section, we know that $u_{\Phi^{-1}}$ must be of the form  $$u_{\Phi^{-1}}=cY_{0,1}+\sum_{i=1}^{n+1}c_i Y_{1,i}.$$ So the
dimension of the solution space is $n+1$. At the same time, the dimension of the tangent space is also $n+1$, which completes the conclusion
$$\text{span}\{v_{1,\xi_0},\partial_\xi v_{c,\xi}|_{(c_0, \xi_0)} \}=\text{span} \{J_{\Phi}^{1/2}Y_{0,1}, J_{\Phi}^{1/2}Y_{1,i}\circ \Phi,~~i=1,\cdots,n+1\}.$$

\end{proof}

\section{Stability of the log-Sobolev inequality on the sphere $\mathbb{S}^n$}

In this section, we will give the proof of the stability of the log-Sobolev inequality on the sphere $\mathbb{S}^n$. First, we will solve the following eigenvalue problem
\begin{align}\label{eigenvalue problem}
H(u)=\lambda C_n\left[u-\int_{\mathbb{S}^n}ud\omega\frac{1}{|\mathbb{S}^n|}\right], ~~u\in \mathcal{D_{L}}.
\end{align}
We have the following
\begin{lemma}
The eigenvalues of (\ref{eigenvalue problem}) are given by
$$\lambda_l=\frac{n}{2}\left(\frac{\Gamma^{\prime}(n/2+l)}{\Gamma(n/2+l)}-\frac{\Gamma^{\prime}(n/2)}{\Gamma(n/2)}\right),~~l=1,2,\cdots.$$
And the eigenfunctions corresponding to the eigenvalue $\lambda_l$ are the spherical harmonics of degree $l$ and constant functions.
\end{lemma}
\vskip0.5cm
\begin{proof}
For $s > 0$,  we define the operator $\mathcal P_{2s}$ given by
\begin{align}\label{definition of P}
\mathcal{P}_{2s}(u)(\xi) := \frac{\Gamma(\frac{n-2s}{2})}{2^{2s}\pi^{\frac{n}{2}}\Gamma(s)}\int_{\mathbb{S}^n}\frac{ u(\eta)}{|\xi - \eta|^{n-2s}} d \eta.
\end{align}
By the Funk-Hecke formula (see \cite[Eq.~(17)]{Be1993} and also \cite[Corollary 4.3]{FrLi2012}), we have
$$\mathcal{P}_{2s}(Y_{l,m})=\frac{\Gamma(l+n/2-s)}{\Gamma(l+n/2+s)}Y_{l,m}.$$
Expanding $u \in L^2(\mathbb{S}^n)$ in terms of the spherical harmonics,
\begin{equation}
\label{sph harm exp}
u=\sum_{l=0}^{\infty}\sum_{m=1}^{N(n,l)} u_{l,m}Y_{l,m} \qquad \qquad \textrm{with} \qquad u_{l,m}=\int_{\mathbb{S}^n}uY_{l,m} d \omega,
\end{equation}
then we have the representation
\begin{equation}\label{P2s} \mathcal{P}_{2s} u=\sum_{l=0}^{\infty}\sum_{m=1}^{N(n,l)} u_{l,m}\frac{\Gamma(l+n/2-s)}{\Gamma(l+n/2+s)}Y_{l,m}.
\end{equation}
In passing, we note that the right hand side of \eqref{P2s} is equal to $A_{2s}^{-1} u$ with the operator $A_{2s}$ defined in \eqref{eq:opas}.
Now, we are in the position to give the action of $H$ on $L^2$ functions with spherical harmonics expansion (\ref{sph harm exp}), which follows from the identity
\begin{align}\label{Hu}\nonumber
H(u)(\xi)&=\lim_{s \to 0} \int_{\mathbb{S}^n} \frac{u(\xi) - u(\eta)}{|\xi - \eta|^{n-2s}} d \eta  \\ \nonumber
& =\lim_{s\rightarrow 0}\frac{2\pi^{n/2}}{\Gamma(n/2)}\frac{1}{2s}\left(\frac{\Gamma(n/2-s)}{\Gamma(n/2+s)}-\mathcal{P}_{2s}\right)u\\
& =\frac{2\pi^{n/2}}{\Gamma(n/2)}\sum_{l=0}^{\infty}\sum_{m=1}^{N(n,l)} u_{l,m}\left(\frac{\Gamma^{\prime}(n/2+l)}{\Gamma(n/2+l)}-\frac{\Gamma^{\prime}(n/2)}{\Gamma(n/2)}\right)Y_{l,m}.
\end{align}
On the other hand,  a direct calculation gives
\begin{align*}
\lambda C_n\left[u-\int_{\mathbb{S}^n}ud\omega\frac{1}{|\mathbb{S}^n|}\right]&=\lambda\frac{4\pi^{n/2}}{n\Gamma(n/2)}[\sum_{l=0}^{\infty}\sum_{m=1}^{N(n,l)} u_{l,m}Y_{l,m}-u_{0,1}Y_{0,1}]\\ & =\lambda\frac{4\pi^{n/2}}{n\Gamma(n/2)}[\sum_{l=1}^{\infty}\sum_{m=1}^{N(n,l)} u_{l,m}Y_{l,m}].
\end{align*}
By the orthogonality of spherical harmonics of different degree, the eigenvalue $\lambda_l$ must satisfy
$$\lambda_l=\frac{n}{2}\left(\frac{\Gamma^{\prime}(n/2+l)}{\Gamma(n/2+l)}-\frac{\Gamma^{\prime}(n/2)}{\Gamma(n/2)}\right),~~l=1,2,\cdots,$$
and the corresponding eigenfunctions are spherical harmonics $Y_{l,m}$. This concludes the proof of our lemma.
\end{proof}

We are now ready to give the proof of our  Theorem \ref{Main Theorem 1} and Theorem \ref{thm2}.

\begin{proof}
[Proof of Theorem \ref{Main Theorem 1}]

First we will prove the sharp local stability of the log-Sobolev inequality on the sphere $\mathbb{S}^n$. Namely,
\begin{align}\label{loc l-S ine}
LS(u)\geq \frac{8\pi^{n/2}}{\Gamma(n/2)(n+2)}d^2(u,M)+o(d^{2}(u,M))~~\text{for all}~~u\in {\mathcal D}_{L}.
\end{align}
Take $u\in \mathcal{D}_{L}$ such that $d(u,M)$ is small enough. Since
$$d^2(u,M)=\inf_{c,\xi}\left\langle u-v_{c,\xi}, u-v_{c,\xi} \right\rangle,$$
it is easy to prove that the infimum is attained at a point $(c_0,\xi_0)\in (\mathbb{R}^+,B^{n+1})$. Then it follows that
$u-v_{c_0,\xi_0}\bot TM_{v_{c_0,\xi_0}}$ and we can rewrite $u$ as $u=v_{c_0,\xi_0}+d v$ with $\left\langle v, v \right\rangle=1$ and $d=d(u,M)$. Since the optimizer $v_{c_0,\xi_0}$ satisfies
 $$v_{c_0,\xi_0}=c_0J_{\Phi}^{1/2}$$
for some conformal transformation $\Phi$, then
$$u_{\Phi^{-1}}=c_0+dv_{\Phi^{-1}}$$ and $\left\langle v_{\Phi^{-1}}, v_{\Phi^{-1}} \right\rangle=1$.
Due to the functional $LS(u)$ being conformal invariant under the transformation $u\rightarrow u_{\Phi^{-1}}$ and  of class $\mathcal{C}^2$ on $\mathcal{D}_{L}\setminus\{0\}$, then
\begin{align}\label{LG}\nonumber
LS(u)&=LS(u_{\Phi^{-1}})\\
& =LS(c_0)+\frac{d}{dt}\Big|_{t=0}LS(c_0+tdv_{\Phi^{-1}})+\frac{1}{2}\frac{d^2}{dt^2}\Big|_{t=0}LS(c_0+tdv_{\Phi^{-1}})+o(d^2).
\end{align}
Careful computations yield
\begin{align}\label{first derivative}
LS(c_0)=\frac{d}{dt}\Big|_{t=0}LS(c_0+tdv_{\Phi^{-1}})=0,
\end{align}
\begin{align}\label{second derivative}
\frac{d^2}{dt^2}\Big|_{t=0}LS(c_0+tdv_{\Phi^{-1}})=4d^2\left\langle v_{\Phi^{-1}}, H(v_{\Phi^{-1}})-C_n v_{\Phi^{-1}}+\frac{C_n}{|\mathbb{S}^n|}\int_{\mathbb{S}^n}v_{\Phi^{-1}}d\omega \right\rangle.
\end{align}
Recall that we have already proved  in Lemma \ref{cha of tan spa} that the tangent space $$TM_{v_{c_0,\xi_0}}=\text{span}\{v_{1,\xi_0},\partial_\xi v_{c_0,\xi_0} \}$$
is equal to the space which is generated by the conformal transformation of $Y_{0,1}$ and $Y_{1,i}$, i.e. $\text{span} \{J_{\Phi^{-1}}^{1/2},J_{\Phi^{-1}}^{1/2}Y_{1,i}\circ \Phi^{-1},~~i=1,\cdots,n+1\}$. Since $$v\bot TM_{v_{c_0,\xi_0}}=\text{span} \{J_{\Phi^{-1}}^{1/2},J_{\Phi^{-1}}^{1/2}Y_{1,i}\circ \Phi^{-1},~~i=1,\cdots,n+1\},$$ then
$v_{\Phi^{-1}}\bot\text{span}\{Y_{0,1}, Y_{1,i}, ~~i=1,\cdots,n+1\}$. Thus, it follows that $v_{\Phi^{-1}}$ must have the spherical harmonics expansion
\begin{align}\label{harmonic expansion}
v_{\Phi^{-1}}=\sum_{l=2}^{\infty}\sum_{m=1}^{N(n,l)}c_{l,m}Y_{l,m}.
\end{align}
Therefore, by (\ref{LG}), (\ref{first derivative}), (\ref{second derivative}) and (\ref{harmonic expansion}) we can obtain
\begin{align}\label{computation of LG}\nonumber
LS(u)& = 2d^2\left\langle v_{\Phi^{-1}}, H(v_{\Phi^{-1}})-C_n v_{\Phi^{-1}}+\frac{C_n}{|\mathbb{S}^n|}\int_{\mathbb{S}^n}v_{\Phi^{-1}}d\omega \right\rangle+o(d^2)\\\nonumber
& = 2d^2\frac{4\pi^{n/2}}{n\Gamma(n/2)}\sum_{l=2}^{\infty}\sum_{m=1}^{N(n,l)}
\left(\frac{n}{2}(\frac{\Gamma^{\prime}(n/2+l)}{\Gamma(n/2+l)}-\frac{\Gamma^{\prime}(n/2)}{\Gamma(n/2)})-1\right)c^2_{l,m}+o(d^2)\\\nonumber
& \geq d^2\frac{8\pi^{n/2}}{n\Gamma(n/2)}\left(\frac{n}{2}(\frac{\Gamma^{\prime}(n/2+2)}{\Gamma(n/2+2)}-\frac{\Gamma^{\prime}(n/2)}{\Gamma(n/2)})-1\right)
\sum_{l=2}^{\infty}\sum_{m=1}^{N(n,l)}c^2_{l,m}+o(d^2)\\
& =d^2\frac{8\pi^{n/2}}{\Gamma(n/2)(n+2)}+o(d^2).
\end{align}

Next we will prove that the constant $\frac{8\pi^{n/2}}{\Gamma(n/2)(n+2)}$ in (\ref{loc l-S ine}) is optimal. To this end,   we only need to construct a suitable test function sequence $\{u_k\}_k$ such that $\lim\limits_{k\rightarrow+\infty}\frac{LS(u_k)}{d^2(u_k, M)}=\frac{8\pi^{n/2}}{\Gamma(n/2)(n+2)}$. In fact, we can choose
test functions $1+\epsilon Y_{2,m}$. The same calculations as done in (\ref{computation of LG}) directly lead to
$$LS(1+\epsilon Y_{2,m})=\frac{8\pi^{n/2}}{\Gamma(n/2)(n+2)}\epsilon^2+o(\epsilon^2).$$
On the other hand, since $Y_{2,m}\bot TM_{v_{1,0}}(v_{1,0}=1)$, then $d(1+\epsilon Y_{2,m},M)=\epsilon$ for sufficiently small $\epsilon$ and
$$\frac{LS(1+\epsilon Y_{2,m})}{d^2(1+\epsilon Y_{2,m},M)}\rightarrow \frac{8\pi^{n/2}}{\Gamma(n/2)(n+2)}~~\text{as}~~\epsilon\rightarrow 0.$$ This proves that the constant $\frac{8\pi^{n/2}}{\Gamma(n/2)(n+2)}$ is optimal.
\medskip

\end{proof}

Next, we claim that the optimal constant of global stability of log-Sobolev inequality must be strictly smaller than $\frac{8\pi^{n/2}}{\Gamma(n/2)(n+2)}$. Namely, we will give the proof of Theorem \ref{thm2}.
 
\begin{proof}
[Proof of Theorem \ref{thm2}]
 We adopt the idea of Konig's in \cite{Ko} to do the third-order Taylor expansion to the functional $LS$.
For fixed $y_2\in Y_2$ with
$\|y_2\|_{L^2}=1$, let $U_\epsilon=1+\epsilon y_2$. Direct computations yields
\begin{equation}\begin{split}
LS(1+\epsilon y_2)&=LS(1)+\frac{d}{dt}\Big|_{t=0}LS(1+t\epsilon y_2)+\frac{1}{2}\frac{d^2}{dt^2}\Big|_{t=0}LS(1+t\epsilon y_2)\\
&\ \ +\frac{1}{6}\frac{d^2}{dt^2}\Big|_{t=0}LS(1+t\epsilon y_2)+o(\epsilon^3),
\end{split}\end{equation}
where \begin{align}
LS(1)=\frac{d}{dt}\Big|_{t=0}LS(1+t\epsilon y_2)=0,
\end{align}
\begin{equation}\begin{split}
\frac{1}{2}\frac{d^2}{dt^2}\Big|_{t=0}LS(1+t\epsilon y_2)&=4\epsilon^2\left\langle y_2, H(y_2)-C_n y_2+\frac{C_n}{|\mathbb{S}^n|}\int_{\mathbb{S}^n}y_2d\omega \right\rangle\\
&=\frac{8\pi^{n/2}}{\Gamma(n/2)(n+2)}\epsilon^2,
\end{split}\end{equation}

\begin{equation}\begin{split}
\frac{1}{6}\frac{d^3}{dt^3}\Big|_{t=0}LS(1+t\epsilon y_2)=-\frac{2}{3}C_n\epsilon^3\int_{S^n}y_2^3d\sigma.
\end{split}\end{equation}

According to the definition of the global stability of log-Sobolev inequality, we obtain
\begin{equation}\begin{split}
\inf_{u\in {\mathcal D}_{L}} \frac{LS(u)}{d^2(u,M)}&\leq \frac{LS(1+\epsilon y_2)}{d(1+\epsilon y_2,M)}\\
&=\frac{8\pi^{n/2}}{\Gamma(n/2)(n+2)}-\frac{2}{3}C_n\epsilon\int_{S^n}y_2^3d\sigma+o(\epsilon).
\end{split}\end{equation}
In order to show that $$\inf_{u\in {\mathcal D}_{L}} \frac{LS(u)}{d^2(u,M)}<\frac{8\pi^{n/2}}{\Gamma(n/2)(n+2)},$$ we only need to find suitable
$y_2\in Y_2$ such that $\int_{S^n}y_2^3d\sigma>0$. In fact, pick $\tilde{y}_2=w_1w_2+w_2w_3+w_3w_1$, then we have
$$\tilde{y}_2^3=I_1(w)+I_2(w)+I_3(w),$$
where
\begin{equation}\begin{split}
&I_1(w)=6w_1^2w_2^2w_3^2\\
&I_2(w)=6(w_1w_2^2w_3^2+w_2w_1^2w_3^2+w_3w_1^2w_2^2)\\
&I_3(w)=w_1^3w_2^3+w_2^3w_3^3+w_1^3w_3^3.
\end{split}\end{equation}
Observing that all the monomials in $I_2$ and $I_3$ are odd function of at least one coordinate $w_i$ (i=1,2, 3), hence
$$\int_{S^n}I_2(w)d\sigma=\int_{S^n}I_3(w)d\sigma=0.$$
Picking $y_2=\frac{\tilde{y}_2}{\|\tilde{y}_2\|_{L^2}}$, we derive $\int_{S^n}y_2^3d\sigma=\|\tilde{y}_2\|_{L^2}^{-1}\int_{S^n}6w_1^2w_2^2w_3^2d\sigma>0$.

\vskip 0.1cm

\end{proof}

Finally, we will explain why we cannot expect the following stable log-Sobolev inequality using double integral as distance:
for any $u\in\mathcal D_{L}$, there holds
$$LS(u)\geq c d^2_{0}(u,M)$$
for some $c>0$, where $$d^2_{0}(u,v)=\iint\limits_{\mathbb{S}^n\times \mathbb{S}^n}\frac{[u(\omega)-v(\omega)-u(\eta)+v(\eta)]^2}{|\omega-\eta|^n}d \omega d \eta.$$

\begin{proof}
[Proof of Theorem \ref{Main Theorem 2}]
First, we claim that for any $u\in \mathcal D_{L}$, there exist a $v_{c_0,\xi_0}$ in the optimizer set $M$ such that
$$d_0(u,v_{c_0,\xi_0})=d_0(u,M).$$
Observing that $d_0(u,v_{c,\xi})$ is continuous with respect to the variables $(c,\xi)\in \mathbb{R}\times B^{n+1}$, in order to prove that $d_0(u,M)$ can be achieved by some function $v_{c_0,\xi_0}$, we only need to show that $$\lim\limits_{c\rightarrow+\infty}d_0(u,v_{c,\xi} )=+\infty,\ \
\lim\limits_{\xi\rightarrow \partial B^{n+1}}d_0(u,v_{c,\xi} )=+\infty.$$ Clearly,
$$\lim\limits_{c\rightarrow +\infty}d_0(u,v_{c,\xi} )\geq \lim\limits_{c\rightarrow +\infty}|d_0(v_{c,\xi},0)-d_0(u,0)|=\lim\limits_{c\rightarrow +\infty}|cd_0(v_{1,\xi}, 0)-d_0(u,0)|=+\infty.$$
On the other hand, by (\ref{Hu}) and Parseval's identity, there holds
$$d_0(v_{c,\xi}, 0)\geq C_0\|v_{c,\xi}-\overline{v_{c,\xi}}\|_{L^2(\mathbb{S}^n)}.$$
for some constant $C_0>0$, where $\overline{v_{c,\xi}}=\int_{\mathbb{S}^n}v_{c,\xi} d\omega/|\mathbb{S}^n|$.
A direct calculation gives that $$\lim_{\xi\rightarrow \partial B^{n+1}}\int_{\mathbb{S}^n}v_{c,\xi} d\omega\lesssim \lim_{\xi\rightarrow \partial B^{n+1}}\int_{\mathbb{S}^n}|\xi-w|^{-\frac{n}{2}}d\omega\lesssim 1.$$
Applying Fatou's Lemma, we directly derive that $$\lim\limits_{\xi\rightarrow \partial B^{n+1}}\|v_{c,\xi}-\overline{v_{c,\xi}}\|_{L^2(\mathbb{S}^n)}^2
\geq \int_{\mathbb{S}^n}\lim\limits_{\xi\rightarrow \partial B^{n+1}}|v_{c,\xi}-\overline{v_{c,\xi}}|^2d\omega=+\infty,$$ which together with $d_0(u,v_{c,\xi})\geq |d_0(v_{c,\xi},0)-d_0(u,0)|$ and
$d_0(v_{c,\xi}, 0)\geq C_0\|v_{c,\xi}-\overline{v_{c,\xi}}\|_{L^2(\mathbb{S}^n)}$ yields that
$$\lim\limits_{\xi\rightarrow \partial B^{n+1}}d_0(u,v_{c,\xi} )=+\infty.$$ This accomplies the proof of the existence of an optimizer.
\vskip 0.1cm

Now choose $1+\epsilon Y_{1,m}$,  by the same computation as done in (\ref{computation of LG}) we have
$$LS(1+\epsilon Y_{1,m})=o(\epsilon^2).$$
On the other hand, by the early proved claim that for any $u\in \mathcal D_{L}$, there exist a $v_{c_0,\xi_0}$ in the optimizer set $M$ such that
$$d_0(u,v_{c_0,\xi_0})=d_0(u,M),$$
there exist a $v_{c_0,\xi_0}$ such that $d_0(Y_{1,M},v_{c_0,\xi_0})=d_0(Y_{1,m},M)$, which implies $d_0(Y_{1,m},M)\neq 0$. Otherwise $Y_{1,m}=v_{c_0,\xi_0}+c$, which is impossible according to the definition of
$Y_{1,m}$. Then
$$d^{2}_{0}(1+\epsilon Y_{1,m},M)=\epsilon^2 d^2_0(Y_{1,m},M).$$
Hence, we derive that
$$\lim\limits_{\epsilon\rightarrow 0}\frac{LS(1+\epsilon Y_{1,m})}{d^{2}_{0}(1+\epsilon Y_{1,m},M)}=\lim\limits_{\epsilon\rightarrow 0}\frac{o(\epsilon^2)}{\epsilon^2 d^2_0(Y_{1,m},M)}=0,$$ which completes our proof.
\end{proof}

\section{Stability of the Moser-Onofri inequality on the sphere $\mathbb{S}^n$}

In this section, we will prove the local stability of the Moser-Onofri inequality (\ref{M-T ine 1}):
\begin{equation}\label{MOstab}
MO(u)\geq \frac{\Gamma(n+1)}{2^{n-1}\pi^{n/2}\Gamma(n/2)} d^2(u,M)+o(d^2(u,M))
\end{equation}
and show that the global version of this type of stable inequality does not hold.
\vskip 0.2cm

\begin{proof}
[Proof of Theorem \ref{Main Theorem 3}]
As we discussed in Section 3, for any $u\in \mathcal Q_{M}$, there exists a function $v_{c_0,\xi_0}$ in $M$ satisfying $u=v_{c_0,\xi_0}+d v$, where $\|v\|_{L^2}=1$, $v\perp TM_{v_{c_0,\xi_0}}$ and $v_{c_0,\xi_0}=c_0J_{\Phi}^{1/2}$ for some
conformal map on $\mathbb{S}^n$. Recall $u_{\Phi^{-1}}=J_{\Phi^{-1}}^{1/2}u\circ\Phi^{-1}$, thus
$$u_{\Phi^{-1}}=c_0+dv_{\Phi^{-1}}.$$
Since the functional $MO(u)$ is conformal invariant under the transformation $u\rightarrow u_{\Phi^{-1}}$ and of class $\mathcal{C}^2$ on $\mathcal{Q}_{M}\setminus\{0\}$, then
\begin{align}\label{MO}\nonumber
MO(u)&=MO(u_{\Phi^{-1}})\\
& =MO(c_0)+\frac{d}{dt}\Big|_{t=0}MO(c_0+tdv_{\Phi^{-1}})+\frac{1}{2}\frac{d^2}{dt^2}\Big|_{t=0}MO(c_0+tdv_{\Phi^{-1}})+o(d^2).
\end{align}
Direct computations give
\begin{align}\label{first derivative1}
MO(c_0)=\frac{d}{dt}\Big|_{t=0}MO(c_0+tdv_{\Phi^{-1}})=0,
\end{align}
\begin{align}\label{second derivative1}\nonumber
& \frac{d^2}{dt^2}\Big|_{t=0}MO(c_0+tdv_{\Phi^{-1}})\\
& \ \ =4d^2\left[\frac{1}{n!}\int_{\mathbb{S}^n} \!\!\!\!\!\!\!\!\!\!\!\!\; {}-{} \,\,\, v_{\Phi^{-1}}\big(A_n(v_{\Phi^{-1}})\big)d \omega
-\int_{\mathbb{S}^n} \!\!\!\!\!\!\!\!\!\!\!\!\; {}-{} \,\,\, v_{\Phi^{-1}}^2 d \omega+(\int_{\mathbb{S}^n} \!\!\!\!\!\!\!\!\!\!\!\!\; {}-{} \,\,\, v_{\Phi^{-1}} d \omega)^2\right].
\end{align}
Since we have already proved that
$$v\bot TM_{v_{c_0,\xi_0}}=\text{span} \{J_{\Phi^{-1}}^{1/2},J_{\Phi^{-1}}^{1/2}Y_{1,i}\circ \Phi^{-1},~~i=1,\cdots,n+1\},$$ then
$v_{\Phi^{-1}}\bot\text{span}\{Y_{0}, Y_{1,i}, ~~i=1,\cdots,n+1\}$. Thus $v_{\Phi^{-1}}$ has the spherical harmonics expansion
$$v_{\Phi^{-1}}=\sum_{l=2}^{\infty}\sum_{m=1}^{N(n,l)}c_{l,m}Y_{l,m}.$$
Using this fact with (\ref{MO}),(\ref{first derivative1}),(\ref{second derivative1}) and $\|v_{\Phi^{-1}}\|_{L^2}=1$ we can obtain
\begin{align}\label{computation of MO}\nonumber
MO(u) &= 2d^2\left[\frac{1}{n!}\int_{\mathbb{S}^n} \!\!\!\!\!\!\!\!\!\!\!\!\; {}-{} \,\,\, v_{\Phi^{-1}}(A_n v_{\Phi^{-1}}) d \omega
-\int_{\mathbb{S}^n} \!\!\!\!\!\!\!\!\!\!\!\!\; {}-{} \,\,\, v_{\Phi^{-1}}^2 d \omega+(\int_{\mathbb{S}^n} \!\!\!\!\!\!\!\!\!\!\!\!\; {}-{} \,\,\, v_{\Phi^{-1}} d \omega)^2\right]+o(d^2)\\\nonumber
& = \frac{2}{|\mathbb{S}^n|}d^2\sum_{l=2}^{\infty}\sum_{m=1}^{N_{n,l}}
\left(\frac{l(l+1)\cdots(l+n-1)}{n!}-1\right) c^2_{l,m}+o(d^2)\\
& \geq \frac{\Gamma(n+1)}{2^{n-1}\pi^{n/2}\Gamma(n/2)}d^2
\sum_{l=2}^{\infty}\sum_{m=1}^{N_{n,l}}c^2_{l,m}+o(d^2) =\frac{\Gamma(n+1)}{2^{n-1}\pi^{n/2}\Gamma(n/2)}d^2+o(d^2).
\end{align}

To prove that the constant $\frac{\Gamma(n+1)}{2^{n-1}\pi^{n/2}\Gamma(n/2)}$ is optimal, we choose the test functions $1+\epsilon Y_{2,m}$. By the same calculation  as done in (\ref{computation of MO}), we have
$$MO(1+\epsilon Y_{2,m})=\epsilon^2 \frac{\Gamma(n+1)}{2^{n-1}\pi^{n/2}\Gamma(n/2)} +o(\epsilon^2).$$
On the other hand, since $Y_{2,m}\perp TM_{1}$, then $d(1+\epsilon Y_{2,m},M)=\epsilon$ for small $\epsilon$. Then it follows that
$$MO(1+\epsilon Y_{2,m})=\frac{\Gamma(n+1)}{2^{n-1}\pi^{n/2}\Gamma(n/2)} d^2(1+\epsilon Y_{2,m},M)+o(d^2(1+\epsilon Y_{2,m},M)),$$
which implies the sharpness of the constant $\frac{\Gamma(n+1)}{2^{n-1}\pi^{n/2}\Gamma(n/2)}$.
\medskip

Finally, since the functional $MO(u)$ is $0$-homogeneous, by Theorem (\ref{l-sta imply g-sta}), we deduce that there does not exist the global version of the local stable Moser-Onofri inequality \eqref{MOstab} on the sphere $\mathbb{S}^n$.

\end{proof}

\bibliographystyle{amsalpha}

\end{document}